\documentclass[12pt,twoside,reqno,psamsfonts]{amsart}

\usepackage{graphicx}
\usepackage{amsmath}
\usepackage{amssymb}
\usepackage{enumerate}
\usepackage{verbatim}
\usepackage{url}
\usepackage[colorlinks = true, citecolor = black, linkcolor = black, urlcolor = black, pdfstartview=FitH]{hyperref}  

\numberwithin{equation}{section}
\renewcommand{\AA}{\mathcal{A}}

\newcommand{\R}{\mathbb{R}}
\newcommand{\C}{\mathbb{C}}
\newcommand{\Q}{\mathbb{Q}}
\newcommand{\N}{\mathbb{N}}

\renewcommand{\P}{\mathbb{P}}
\newcommand{\E}{\mathbb{E}}

\newcommand{\cL}{\mathcal{L}}

\newcommand{\cF}{\mathcal{F}}

\newcommand{\cW}{\mathcal{W}}

\newcommand{\cA}{\mathcal{A}}

\newcommand{\eps}{\varepsilon}
\newcommand{\one}{\mathbf{1}}

\newcommand{\Z}{\mathbb{Z}}

\renewcommand{\emptyset}{\varnothing}
\renewcommand{\epsilon}{\varepsilon}
\renewcommand{\rho}{\varrho}
\renewcommand{\phi}{\varphi}

\newcommand{\calF}{{\cal F}}

\renewcommand{\mod}{\,\,\mathrm{mod\,}}

\renewcommand{\hat}{\widehat}
\newcommand{\Rea}{\operatorname{Re}}
\renewcommand{\iint}{\int\hspace{-0.1in}\int}

\DeclareMathOperator{\spt}{supp}

\DeclareMathOperator{\supp}{supp}

\DeclareMathOperator{\Lip}{Lip}

\theoremstyle{plain}
\newtheorem{thm}{Theorem}[section]
\newtheorem{theorem}{Theorem}[section]

\newtheorem{lemma}[thm]{Lemma}
\newtheorem{prop}[thm]{Proposition}

\theoremstyle{definition}

\theoremstyle{remark}
\newtheorem{remark}[thm]{Remark}

\usepackage[usenames,dvipsnames]{xcolor}

\graphicspath{ {pictures/} }



\thanks{TS was partially supported by the Marie Sk{\l}odowska-Curie Individual Fellowship grant $\sharp$655310 and a start-up fund from the School of Mathematics, University of Manchester, UK}

\title{Trigonometric series and self-similar sets}

\author{Jialun Li}
\address{Institut de Math\'ematiques de Bordeaux, Universit\'e de Bordeaux, 351 cours de la Lib\'eration, Talence, France}
\address{Current address: Institut f\"ur Mathematik, Universit\"at Z\"urich, 190 Winterthurerstrasse, Z\"urich, Switzerland}
\email{jialun.li@math.uzh.ch}

\author{Tuomas Sahlsten}
\address{School of Mathematics, Alan Turing Building, University of Manchester, Oxford Road, Manchester, UK}
\email{tuomas.sahlsten@manchester.ac.uk}

\begin{document}

\begin{abstract}Let $F$ be a self-similar set on $\R$ associated to contractions $f_j(x) = r_j x + b_j$, $j \in \cA$, for some finite $\cA$, such that $F$ is not a singleton. We prove that if $\log r_i / \log r_j$ is irrational for some $i \neq j$, then $F$ is a set of multiplicity, that is, trigonometric series are not in general unique in the complement of $F$. No separation conditions are assumed on $F$. We establish our result by showing that every self-similar measure $\mu$ on $F$ is a Rajchman measure: the Fourier transform $\widehat{\mu}(\xi) \to 0$ as $|\xi| \to \infty$. The rate of $\widehat{\mu}(\xi) \to 0$ is also shown to be logarithmic if $\log r_i / \log r_j$ is diophantine for some $i \neq j$. The proof is based on quantitative renewal theorems for stopping times of random walks on $\R$.
\end{abstract}

\maketitle

\section{Introduction and the main result}

The \textit{uniqueness problem} in Fourier analysis that goes back to Riemann \cite{Riemann} and Cantor \cite{Cantor} concerns the following question: suppose we have two converging trigonometric series $\sum a_n e^{2\pi i nx}$ and $\sum b_n e^{2\pi i nx}$ with coefficients $a_n,b_n \in \C$ such that for ``many'' $x \in [0,1]$ they agree:
\begin{align}\label{eq:uniq}\sum_{n \in \Z} a_n e^{2\pi i n x} = \sum_{n \in \Z} b_n e^{2\pi i n x},\end{align}
then are the coefficients $a_n = b_n$ for all $n \in \Z$? For how ``many'' $x \in [0,1]$ do we need to have \eqref{eq:uniq} so that $a_n = b_n$ holds for all $n \in \Z$? If we assume \eqref{eq:uniq} holds \textit{for all} $x \in [0,1]$, then using Toeplitz operators Cantor \cite{Cantor} proved that indeed $a_n = b_n$ for all $n \in \Z$. However, it would be interesting to see how small the set of $x \in [0,1]$ satisfying \eqref{eq:uniq} can be, so that we have $a_n = b_n$ for all $n \in \Z$. Motivated by this one defines that a subset $F \subset [0,1]$ is a \textit{set of uniqueness} if whenever we have coefficients $a_n,b_n \in \C$, $n \in \Z$, such that \eqref{eq:uniq} holds for all $x \in [0,1] \setminus F$, then $a_n = b_n$ for all $n \in \Z$. Here one defines also that if $F$ is not a set of uniqueness, then it is called a \textit{set of multiplicity}. In particular by Cantor's result this shows that the empty set $\emptyset$ is a set of uniqueness and so $[0,1]$ is a set of multiplicity.

Cantor \cite{Cantor} proved that that every closed countable set is a set of uniqueness, and later Young \cite{Young} generalised to every countable set. In the uncountable case, however, even if assuming $F$ is very small, uniqueness of $F$ may fail: Menshov \cite{Menshov} constructed a set $F$ of Lebesgue measure $0$, which is a set of multiplicity, that is, the uniqueness problem fails if we only assume \eqref{eq:uniq} for all $x \in [0,1] \setminus F$. This can be proved using the following criteria, which goes back to Salem \cite{Salem} that if a set $F$ supports a Borel probability measure $\mu$ such that the \textit{Fourier transform}
$$\widehat{\mu}(\xi) := \int e^{-2\pi i \xi x} \, d\mu(x), \quad \xi \in \R,$$
satisfies $\widehat{\mu}(n) \to 0$ as $|n| \to \infty$, $n \in \Z$, then $F$ is a set of multiplicity. Such measures $\mu$ are called \textit{Rajchman measures} in the literature. Hence constructing measures $\mu$ with decaying Fourier coefficients provides a way to check whether $F$ is of multiplicity. It remains an open problem to classify which uncountable sets $F$ are of multiplicity and which $F$ are of uniqueness and much work has been done in many examples of $F$ on trying to establish their uniqueness or multiplicity, see for example the works of Kechris \textit{et al.} \cite{KL92} on connections to descriptive set theory.

In the series of works Salem \cite{Salem} proved that the middle third Cantor set $C_{1/3}$ is a set of uniqueness. More generally, Salem established that if $C_\lambda$ is the middle $\lambda$-Cantor set with $0 < \lambda < 1/2$, that is, interval of length $1-2\lambda$ is removed from the center of $[0,1]$ at every construction stage, then $C_\lambda$ is a set of uniqueness when $\lambda^{-1}$ is a Pisot number. In the opposite case, if $\lambda^{-1}$ is not a Pisot number, by constructing a Rajchman measure on $C_\lambda$, Piatetski-Shapiro \cite{PS}, Salem and Zygmund \cite{SZ} established that $C_\lambda$ is a set of multiplicity. 

The Cantor set $C_\lambda$ is an example of a \textit{self-similar set}. Recall that a subset $F \subset [0,1]$ is self-similar if there exists similitudes $f_j : [0,1] \to [0,1]$, that is, $f_j(x) = r_j x + b_j$, $j \in \cA$, for some finite set $\cA$, translations $b_j \in \R$ and contractions $0 < r_j < 1$ such that
$$F = \bigcup_{j \in \cA} f_j (F).$$
As far as we know nothing is known about the uniqueness or multiplicity of self-similar sets beyond the case of $C_\lambda$ or if adding finitely many more similitudes with the same contraction ratio $\lambda$ to the definition, which was done by Salem \cite{Salem}. For example if we have two different contractions $r_0 = 1/2$ and $r_1 = 1/3$ for the iterated function system, do we expect $F$ to be of multiplicity or of uniqueness? Due to having same contraction ratio $\lambda$ the case $C_\lambda$  has a convolution structure, which is helpful when connecting to the algebraic properties of the number $\lambda$. In the general case, however, we would need to find a way out of this.

It turns out that the algebraic properties of the additive subgroup $\Gamma$ generated by the log-contraction ratios $\{-\log r_j : j \in \cA\}$ in $\R$ is important in the study of the multiplicity of a self-similar set $F$ with contraction ratios $r_j$. In particular if this subgroup $\Gamma$ is dense, which happens when $\log r_j / \log r_\ell$ is irrational for some $j \neq \ell$ (e.g. $r_j = 1/2$ and $r_\ell = 1/3$), we can establish $F$ is a set of multiplicity.

\begin{theorem}\label{thm:multi}
Let $F \subset [0,1]$ be a self-similar set associated to contractions $f_j(x) = r_j x + b_j$, $j \in \cA$, such that $F$ is not a singleton. If $\log r_j / \log r_\ell$ is irrational for some $j \neq \ell$, then $F$ is a set of multiplicity.
\end{theorem}

Notice that by assuming $\log r_j / \log r_\ell$ is irrational we exclude the case of $C_\lambda$ as in that case every ratio of logarithms of the contractions is just $1$. It remains an open problem to study the case when $\log r_j / \log r_\ell \in \Q$ for all $j \neq \ell$. We predict that here typically $F$ should be a set of uniqueness unless all the contraction ratios are equal, like the case $C_\lambda$, and then an algebraic number theoretic condition like $\lambda^{-1}$ being Pisot needs to be imposed.

In order to prove the multiplicity of a self-similar set $F$, it is enough by Salem's criterion \cite{Salem} for multiplicity to find a Rajchman measure supported on $F$. Hence Theorem \ref{thm:multi} follows by establishing that all positive dimensional self-similar measures on $F$ are Rajchman measures. Recall that a probability measure $\mu$ on $\R$ is called \textit{self-similar} if there exists a finite collection $\{f_j : j \in \cA\}$ of similitudes of $\R$ with at least two maps and weights $0 < p_j < 1$, $j \in \cA$, with $\sum_{j \in \cA} p_j = 1$  such that $\mu = \sum_{j \in \cA} p_j f_j \mu$.

\begin{theorem}\label{thm:main} Let $F \subset [0,1]$ be a self-similar set associated to contractions $f_j(x) = r_j x + b_j$, $j \in \cA$, such that $F$ is not a singleton. If $\log r_j / \log r_\ell$ is irrational for some $j \neq \ell$, then the Fourier transform $\widehat{\mu}(\xi) \to 0$ as $|\xi| \to \infty$ for every self-similar measure $\mu$ on $F$.
\end{theorem}

Theorem \ref{thm:main} is closely related to another currently active problem in the community of fractal geometry, where we would like to understand the Fourier transforms of fractal measures, see the book \cite{Mattila} by Mattila for an history and overview. In particular there are various past and recent works on random fractals by Kahane \cite{KahaneImage,KahaneLevel}, Shmerkin and Suomala \cite{ShmerkinSuomala} and other people \cite{FOS, FS}, connections to Diophantine approximation by Kaufman \textit{et al.} \cite{Kaufman1, Kaufman2}, dynamical systems \cite{JordanSahlsten, SahlstenStevens} and additive combinatorics \cite{Bourgain2010,LP09}. Analysing the spectrum of fractal measures has been particularly important in finding normal numbers from the support of fractals \cite{HochmanShmerkinEquidistribution,QetR,DEL} and the study of harmonic analysis defined by fractal measures, see for example applications to the spectrum of convolution operators defined by fractal measures in the work of Sarnak \cite{Sarnak} and later by Sidorov and Solomyak \cite{SidorovSolomyak}, and more recently applications to quantum resonances in quantum chaos by Bourgain and Dyatlov \cite{BourgainDyatlov}.

The study of Fourier transforms of self-similar measures in general goes back to the works of Strichartz \cite{Strichartz1,Strichartz2}, where an average decay of Fourier transform $\widehat{\mu}(\xi)$ of self-similar measures $\mu$ was obtained, where proportions of frequencies $\xi \in \R$ are excluded. More recently a large deviation estimate for these average decays was proved by Tsujii \cite{Tsujii}. However, the methods here cannot be used to obtain a full decay over all $|\xi| \to \infty$. Before Theorem \ref{thm:main} the only cases of self-similar measures $\mu$ where $\widehat{\mu}(\xi) \to 0$ as $|\xi| \to \infty$ was known were \textit{Bernoulli convolutions} $\mu_\beta$, $\beta > 1$, which are the distribution of the random sum $\sum \pm \beta^{-k}$ with i.i.d. chosen signs. The case $\beta > 2$ was studied by Piatetski-Shapiro \cite{PS}, Salem and Zygmund \cite{SZ}, where up to a translation $\mu_\beta$ are natural self-similar measures on middle $\beta^{-1}$ Cantor set. For Bernoulli convolutions $\mu_\beta$ with $1 < \beta < 2$ Fourier transforms play an important role. In particular, proving that $\widehat{\mu}_\beta(\xi)$ has sufficiently fast power decay as $|\xi| \to \infty$ implies that $\mu_\beta$ is absolutely continuous, which is a well-known open problem in the field, see for example Shmerkin \cite{ShmerkinGAFA}. It is known by the results of Erd\"os \cite{Erdos} and Kahane \cite{Kahane} that the set of $1 < \beta < 2$ such that $\mu_\beta$ does not have a power decay has Hausdorff dimension zero. Moreover, if $\beta$ is is not a Pisot number, then Salem \cite{Salem} proved $\widehat{\mu}_\beta(\xi) \to 0$ as $|\xi| \to \infty$, and conversely if $\beta$ is a Pisot number, Erd\"os \cite{Erdos} proved that $\widehat{\mu}_\beta(\xi) \not\to 0$ as $|\xi| \to \infty$. In the non-Pisot case the rate of convergence was later shown to be logarithmic for rational number $\beta$ by Kershner \cite{kershner}, see also Dai \cite{Dai1} and Bufetov and Solomyak \cite{Bufetov}, and some power decay for algebraic numbers $\beta$ has been obtained by Dai, Feng and Wang \cite{Dai2}.

Notice that in Theorem \ref{thm:multi} and Theorem \ref{thm:main} there can be any types of overlaps for the maps $f_j$ and no separation conditions are assumed. Typically in the overlapping case the analysis of self-similar sets and measures can be notoriously difficult to understand, say, their Hausdorff dimension has required some deep connections to additive combinatorics, see for example the recent works of Hochman \cite{Hochman}, Breuillard-Varj\'u \cite{BV} and Varj\'u \cite{V}. The reason overlaps do not cause us any issues is the fact that the main contribution to the Fourier decay comes from controlling the distribution of lengths of the construction intervals, and not their relative positions. Understanding the distribution of the lengths of the construction intervals then can be reduced as a problem of studying the renewal theory for stopping times of random variables $X_1,X_2,\dots$ on $\R$ with distribution $\lambda = \sum_{j \in \cA} p_j \delta_{-\log r_j}$. This strategy to establish Fourier decay is inspired by the case of the stationary measure for Lie group actions by the first author in \cite{Li1}. In the self-similar case case we consider, however, the proof is much more straightforward and we can see the idea governing the Fourier decay more clearly. In our case we will prove a quantitative version of Kesten's renewal theorem for stopping time given in \cite{Kesten}, see Section \ref{sec:renewal}. The irrationality of $\log r_i / \log r_j$ is key to prove the random walk becomes \textit{non-lattice}, that is, not concentrated on an arithmetic progression, which is a key assumption for the renewal theorem for stopping times we employ.

If we want a rate of convergence in Theorem \ref{thm:main} using the strategy we present in this paper, one needs to go into the rate of convergence for the renewal theorems we use. Here it is well-known that the diophantine properties of the random walk become an essential property, in particular, how well $\log r_i / \log r_j$ is approximated by rationals. In Diophantine approximation, it is defined that an irrational real number $a \in \R$ is called \textit{diophantine} if for some $c> 0$ and $l>2$ we have
\begin{align}\label{eq:diophantine}\Big|a - \frac{p}{q}\Big| \geq \frac{c}{q^l}\end{align}
for all $p\in \Z$ and $q\in\N^*$. This happens for example when $a = \log 2 / \log 3$ or in general for $a = \log p / \log q$ with $p,q$ coprime, see Baker \cite{Baker}. Having some diophantine $\log r_i / \log r_j$ in the iterated function system imposes the random walk generated by the contractions to quantitatively avoid lattices and then gives quantitative rates for the renewal theorem. Under this condition, we can improve Theorem \ref{thm:main} in the following way:

\begin{theorem}\label{thm:mainquantitative} Let $F \subset [0,1]$ be a self-similar set associated to contractions $f_j(x) = r_j x + b_j$, $j \in \cA$, such that $F$ is not a singleton. If $\log r_i / \log r_j$ is diophantine for some $i \neq j$, then for every self-similar measure $\mu$ on $F$, there exists $\beta > 0$ such that
$$|\widehat{\mu}(\xi)| = O\Big(\frac{1}{|\log |\xi||^{\beta}}\Big), \quad |\xi| \to \infty.$$
\end{theorem}

\begin{remark}
By using a quantitative equidistribution criterion by Davenport-Erd\"os-LeVeque \cite{DEL}, a particular consequence of the logarithmic decay for Fourier transform of a probability measure $\mu$ on $\R$ is that $\mu$ almost every number is normal in every base. Thus in the setting of Theorem \ref{thm:mainquantitative} we have for any self-similar measure $\mu$ on $F$ that $\mu$ almost every number $x \in F$ is normal in every base. This consequence was pointed out in \cite{GMSZ} and we thank Mel Levin for bringing this to our attention.
\end{remark}

Removing the irrationality of ratios of log-contractions ratios makes the random walk $X_1,X_2,\dots$ on $\R$ generated by $\lambda = \sum_{j \in \cA} p_j \delta_{-\log r_j}$ a lattice, that is, concentrated on an arithmetic progression. Then the renewal theorems do not hold anymore in the same form. In fact, for example in the case of middle $1/3$ Cantor measure, the Fourier transform does not even decay at infinity. However, in the case $\beta$ is not Pisot, the Bernoulli convolution $\mu_\beta$ associated to $\beta$ provides examples of a measure where the Fourier transform does decay at infinity, even with polynomial rate for some algebraic $\beta$, but the additive random walk on $\R$ generated by $\log \beta$ is a lattice. Hence it would be interesting to develop the connection to renewal theory further and find a full classification of self-similar sets $F$ which are of uniqueness and which are of multiplicity. After the completion of this manuscript, these problems have been addressed by Br\'emont \cite{Bre19} and Varj\'u-Yu \cite{VY20}. Moreover, the polynomial Fourier decay case was recently obtained for most self-similar measures by Solomyak \cite{Sol19}.

In this paper we consider the self-similar case, but if we impose that the maps $f_j$ to be suitably nonlinear, such as the inverse branches of the Gauss map $x \mapsto 1/x \mod 1$ and study the Fourier transforms of self-conformal measures $\mu$, then the rates of Fourier decay in Theorem \ref{thm:mainquantitative} for Fourier decay can be improved to power decay, see for example the works \cite{JordanSahlsten,BourgainDyatlov,SahlstenStevens,Li2}. Here the non-lattice condition of contractions $-\log r_j$ is replaced by a non-concentration condition of the log-derivatives of the iterates $-\log (f_{j_1} \circ \dots \circ f_{j_n})'(x)$ as $n \to \infty$. These types of conditions appear in the Fourier decay properties of multiplicative convolutions in the discretized sum-product theory developed by Bourgain \cite{Bourgain2010}.

What about the higher dimensional case? Here the analogue to Theorem \ref{thm:main} and Theorem \ref{thm:mainquantitative} would be to understand Fourier transforms $\widehat{\mu}$ of \textit{self-affine measures} $\mu$ on $\R^d$. They are measures on $\R^d$ associated to affine contractions $f_j =  A_j + b_j$ of $\R^d$, $j \in \cA$, for some finite set $\cA$, where $b_j \in \R^d$ and $A_j \in \mathrm{GL}(d,\R)$ such that
$$\mu = \sum_{j \in \cA} p_j f_j \mu$$
 for some weights $0 < p_j < 1$, $j \in \cA$, with $\sum_{j \in \cA} p_j = 1$. In a follow-up paper \cite{LSaffine}, we apply a similar strategy as we do in this paper by considering renewal theory for random walks on the group $\mathrm{GL}(d,\R)$ coming from $\{A_j  : i \in \cA\}$ to establish a Fourier decay for self-affine measures. The renewal theory we need has been done recently by the first author in \cite{Li2}. Here the non-lattice condition can be replaced by an irreducibility and proximality assumption of the subgroup $\Gamma$ generated by $\{A_j : i \in \cA\}$ as B\'ar\'any, Hochman and Rapaport did recently in their work \cite{BHR} for the computation of Hausdorff dimension of self-affine measures on $\R^2$. Moreover, due to the better rates for quantitative renewal theorems for random walks in real split groups \cite{Li2}, that is, when the Zariski closure of $\Gamma$ is $\R$-splitting, we can improve the rates for the Fourier decay of $\mu$ to power decay.\\
 
\textbf{Organisation of the paper.} The article is organised as follows. In Section \ref{sec:renewal} we describe the key method used in the paper and give the quantitative renewal theorems we need for our results and then prove them in Section \ref{secrentheory}. Then in Section \ref{sec:proofs} we give the proof of Theorem \ref{thm:main}, which implies Theorem \ref{thm:multi} on the multiplicity of self-similar sets, and also in Section \ref{sec:proofs} we prove the quantitative Theorem \ref{thm:mainquantitative} using the quantitative estimates for the renewal theorem stated in Section \ref{sec:renewal}.

\section{Renewal theorems for stopping time of random walks in $\R$}\label{sec:renewal}

The main method used to establish Theorems \ref{thm:main} and \ref{thm:mainquantitative} is based on renewal theory. Renewal theory has a long history of research both in probability theory and dynamical systems, see e.g. \cite{Feller, Lalley} and various other works citing them. Here we will need quantitative renewal theorems for stopping time of random walks in $\R$, which we will now describe. 

Suppose $\lambda$ is a probability measure on $\R^+:=\{x\in\R,\ x>0 \}$ with finite support. Let 
$$\sigma := \int x \, d\lambda(x)$$ 
be the expectation of $\lambda$, which is positive by definition. Renewal type results are valid under more general assumption. For simplicity, we state it under this simple assumption which is sufficient for the proof of main results in the manuscript.
Let $X_1,X_2,\dots$ be i.i.d. sequence of random variables in $\R$ with probability distribution $\lambda$. Write for $n \in \N$ the sum
$$S_n := X_1+X_2+\dots+X_n.$$
For a non-negative bounded Borel function $g$ on $ \R$, define the \textit{renewal operator} by
\begin{align*}
Rg(t) = \sum_{n=0}^{\infty}\E (g(S_n-t)), \quad t \in \R.
\end{align*}
Because of the non-negativity of $g$, this sum is well defined. Kesten's \textit{key renewal theorem} \cite{Kesten} considers the limit behaviour of $R g(t)$ as $t \to \infty$. The algebraic properties of the support $\supp \lambda$ of the distribution $\lambda$ will be crucial in the behaviour of $S_{n} - t$ as $t \to \infty$. If $\lambda$ is \textit{non-lattice}, that is, $\supp \lambda$ generates a dense additive subgroup of $\R$, then the limit $Rg(t)$ as $t\rightarrow \infty$ is given by $\frac{1}{\sigma}\int g(x)\,d x$, where $\, dx$ is the Lebesgue measure.

Here we need to study the convergence for randomly stopped processes $S_1,\dots,S_{n_t}$ at a \textit{stopping time} 
$$n_t := \inf\{n \in \N : S_n \geq t\}$$
and see how the residual process $S_{n_t} - t$ behaves as $t \to \infty$.  Kesten's \textit{renewal theorem for stopping time} \cite{Kesten} says that the residue distribution $S_{n_t}-t$ will converge to a distribution absolutely continuous with respect to the Lebesgue measure when $t$ tends to infinity. 


Let $|\supp \lambda|$ be the supremum of the absolute value of the elements in the support of $\lambda$. Define a local $C^1$ norm on the interval $(-1,|\supp\lambda|+1)$ by 
\begin{align}\label{eq:localC1}\|g\|_{C^1} := \sup\{|g(x)|+|g'(x)| : x\in (-1,|\supp\lambda|+1)\}.\end{align}
We will have the following renewal theorem.
\begin{prop}\label{prop:renewal}
If $\lambda$ is a probability measure on $\R^+$ with finite support and non-lattice, then we have for $t > |\supp \lambda|+1$ and a $C^1$ function $g$ on $\R$ the following asymptotics as $t \to +\infty$:
\begin{equation}
	\E (g(S_{n_t}-t))=\frac{1}{\sigma}\int_{\R^+}g(x)p(x)d x+o_t\| g\|_{C^1},
\end{equation}
where $o_t$ tends to zero as $t$ is going to $\infty$. Here $p(x) := \lambda((x,\infty))$ is a piecewise constant function and vanishes when $x$ passes the support of $\lambda$. 
\end{prop}
Proposition \ref{prop:renewal} is equivalent to the classical Kesten's renewal theorem for stopping time \cite{Kesten}. Because the unit ball of $C^1([0,|\supp\lambda|])$ is precompact in $C^0([0,|\supp\lambda|])$, the uniform speed on $C^1$ norm is equivalent to the convergence in distribution of $S_{n_t}-t$, which is exactly Kesten's theorem. We thank one of the anonymous referees for pointing this out to us. 

Proposition \ref{prop:renewal} will be used to prove Theorem \ref{thm:main} later in Section \ref{sec:proofs} and the non-lattice condition for $\lambda$ is obtained using the irrationality of $\log r_j / \log r_\ell$. 
However, in the setting of Theorem \ref{thm:mainquantitative}, where we assume $\log r_j / \log r_\ell$ is diophantine, we will need a quantitative version of Proposition \ref{prop:renewal}, where one needs a stronger assumption on the distribution $\lambda$ called ($l$-)\textit{weakly diophantine}, that is, for some $l > 0$:
	\[\liminf_{|b|\rightarrow \infty}|b|^l|1-\cL\lambda(ib)|>0, \]
	where $\cL\lambda$ is the \textit{Laplace transform} of $\lambda$, defined for $z \in \C$ by the formula
	$$\cL\lambda(z)=\int e^{-zx}d\lambda(x).$$ 
This condition means heuristically that the random walk $X_1,\dots,X_n$ quantitatively avoids concentration on lattices and could be considered as a spectral gap condition for the random walk. The number $l > 0$ will be reflected in the rate of the renewal theorem for stopping time of random walks:

\begin{prop}\label{prop:renewalquantitative}
 If $\lambda$ is $l$-weakly diophantine, then for $t>|\supp \lambda|+1$ we have
\begin{equation}
\E (g(S_{n_t}-t))=\frac{1}{\sigma}\int_{\R^+}g(x)p(x)d x+O(t^{-1/(4l+1)})\|g\|_{C^1}.
\end{equation}
\end{prop}

This polynomial error term is new for the renewal theorem for stopping time. A polynomial error term in the key renewal theorem was obtained by Carlsson in \cite{Carlsson} under the same weakly diophantine hypothesis using Fourier transform. What we do here uses key renewal theorem to prove renewal theorem for stopping time and we track the error term carefully such that the error term in renewal theorem for stopping time can be computed using error term in key renewal theorem.
With stronger hypothesis, that is $\liminf_{|b|\rightarrow \infty}|1-\cL\lambda(ib)|>0$, Blanchet and Glynn give an exponential error term of key renewal theorem in \cite{BG}. However, this stronger hypothesis is never true for finitely supported measures $\lambda$. See \cite{Boyer} for more details and bibliography on the error term of the key renewal theorem. 

In our case of self-similar measures associated to an iterated function system $f_j (x) = r_j x + b_j$ and weights $\sum_{j \in \cA} p_j = 1$, the random walk we will use is given by $X_k = -\log r_{j_k}$, $k \in \N$, where $j_k = j$ with probability $p_j$. Thus
$\lambda = \sum_{j \in \cA} p_j \delta_{-\log r_j}.$
In Section \ref{sec:proofs} we verify the assumptions of the renewal theorems Proposition \ref{prop:renewal} and Proposition \ref{prop:renewalquantitative}. We will explain how they are used to prove the main results on Fourier decay.

\section{Proof of the Fourier decay}\label{sec:proofs}

\subsection{Symbolic notations} 

Let us write $\cA^*$ as the space of all words $w$ with entries in $\cA$ of finite length. Moreover, $\cA^n$ is the space of all words of length $n$ with entries in $\cA$ and $\cA^\infty$ the infinite length words. If $w = w_1w_2\dots w_n \in \cA^n$, define the composition
$$f_w := f_{w_1} \circ \dots \circ f_{w_n}.$$
Then $f_w$ is again a similitude with a contraction
$$r_w := r_{w_1} \dots r_{w_n}.$$
Using this notation the self-similarity of $\mu$ implies that
$$\mu = \sum_{w \in \cA^n} p_w f_w \mu,$$
where
$$p_w := p_{w_1}\dots p_{w_n} > 0$$
as the product of weights $p_j$, $j \in \cA$, according to the entries of the word $w$. See the book by Falconer \cite{Falconer1} for more details, notations and history on self-similar sets and measures.

\subsection{Reduction to exponential sums} \label{sec:reduction} Given $\xi \in \R$ and $t > 0$, the first step is to reduce the Fourier transform of $\mu$ to double $\mu$ integrals over exponential sums determined by a stopping time $n_t$. Recall that we defined in Section \ref{sec:renewal} for $t > 0$ the stopping time
$$n_t := \inf\{n \in \N : S_n \geq t\},$$
where $S_n = X_1 + \dots + X_n$ and $X_j$ are i.i.d. random variables distributed according to
$$\lambda = \sum_{j \in \cA} p_j \delta_{-\log r_j}.$$
We can identify the stopping time $n_t$ as follows using symbolic notations. For an infinite word $\omega = w_1w_2\dots \in \cA^\infty$ write $n_t(\omega)$ as the smallest $n \in \N$ such that the restriction $w := w_1\dots w_n$ satisfies $-\log r_{w} \geq t$, that is $r_w \leq e^{-t}$. Thus in particular $r_w$ will be roughly $e^{-t}$: we have $r_w \in [c e^{-t},e^{-t}]$ for some constant $c > 0$. Then this symbolic definition $n_t(\omega)$ agrees with the stopping time $n_t=\inf\{n \in \N : S_n \geq t\}$ by setting the probability space $(\Omega,\cF,\P)$ with $\Omega = \cA^\infty$, $\cF$ is the Borel $\sigma$-algebra generated by cylinder sets and $\P = (\sum_{j \in \cA} p_j\delta_j)^\infty$. With this in mind, write
$$\cW_t := \{w_1\dots w_{n_t(\omega)} : \omega \in \cA^\infty\}$$
and let $\P_t$ be the probability distribution on $\cA^*$ associated to the stopping time, that is,
$$\P_t := \sum_{w \in \cW_t} p_w \delta_{w}.$$
Thus the support $\spt \P_t = \cW_t$. Now the distribution of $S_{n_t}$ agrees with the distribution of $-\log r_w$, where $w$ follows the distribution of $\P_t$. Then for a continuous function $g$ on $\R$ we obtain
\begin{equation}\label{equ:et}
	\sum_{w \in \cW_t} p_w g(-\log r_w-t)=\int g(-\log r_w-t)\, d\P_t(w)=\E(g(S_{n_t}-t)).
\end{equation}
See Figure \ref{fig1} for an illustration of the words $\cW_t$.

\begin{figure}[ht!]
\includegraphics[scale=0.15]{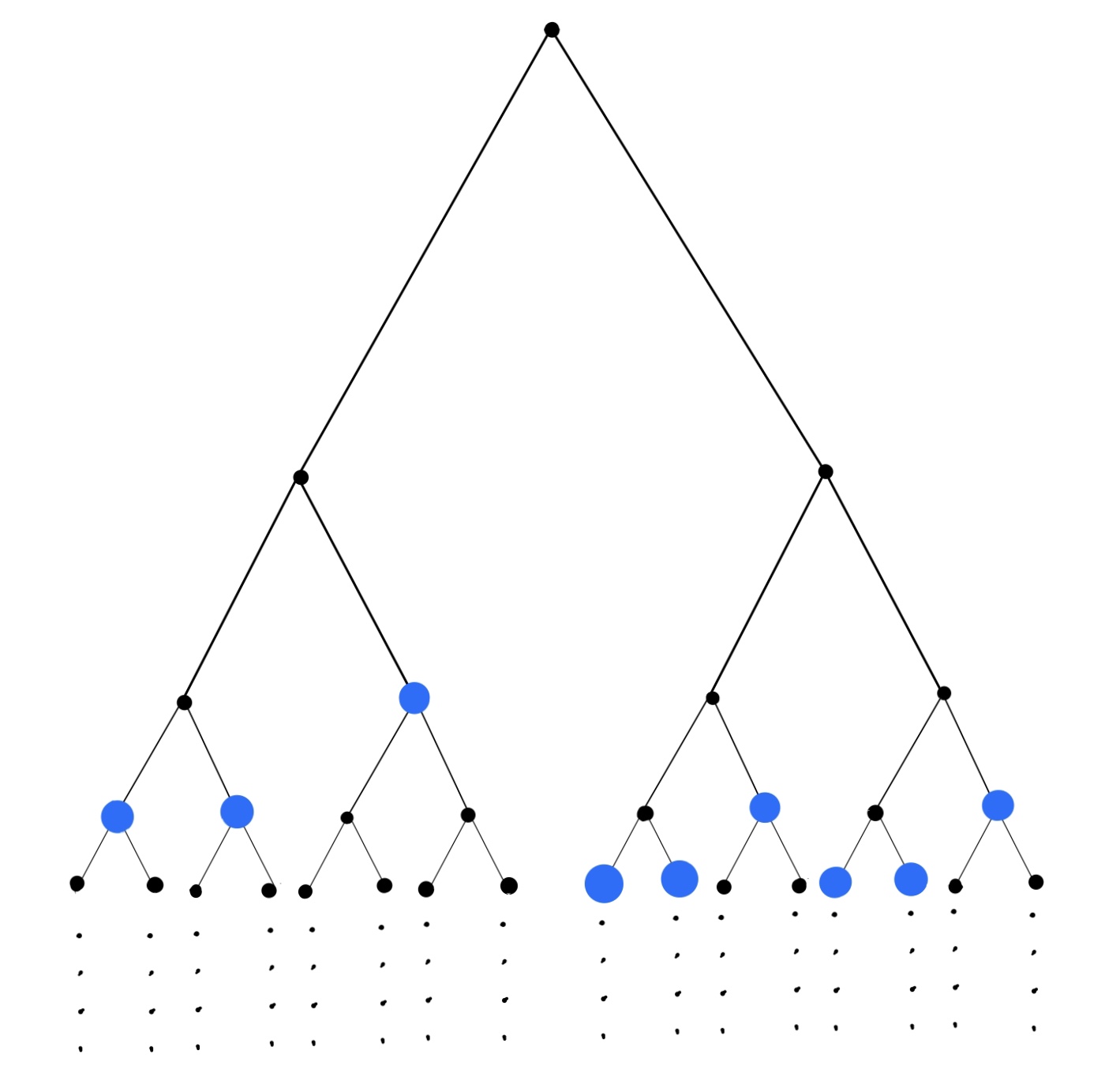}
\caption{Tree $\cA^\infty$ in the case $|\cA| = 2$. The finite blue words define the words in $\cW_t \subset \cA^*$. 
	They are defined to be the shortest words $w \in \cA^*$ with contraction $r_w \leq e^{-t}$. 
}
\label{fig1}
\end{figure}

The reason to use the stopping time here is that we want to use the equidistribution phenomenon of the renewal theorem (Proposition \ref{prop:renewal}), which combined with high-oscillation can give decay of exponential sums.

 Later we will make $t$ depend on $\xi$ and let $|\xi| \to \infty$, but for now we keep everything fixed.

\begin{lemma}\label{lma:estimatefourier}For every $\xi \in \R$ and $t > 0$ we have
$$|\widehat{\mu}(\xi)|^2 \leq \iint \sum_{w \in \cW_t} p_w  e^{-2\pi i \xi (f_w(x)-f_w(y))} \, d\mu(x) \, d\mu(y). $$
\end{lemma}

\begin{proof}
Firstly, by
$$\mu = \sum_{j \in \cA} p_j f_j \mu$$ 
and the martingale stopping theorem, we see that for any $t > 0$,  we can write
$$\mu = \sum_{w \in \cW_t} p_w f_w \mu.$$
(The proof of this is similar to \cite[Proposition 3.5]{Li1}.) Hence we obtain
$$\widehat{\mu}(\xi) = \sum_{w \in \cW_t}  p_w \int e^{-2\pi i \xi f_w(x)} \, d\mu(x).$$
Thus by Cauchy-Schwarz, we have
$$|\widehat{\mu}(\xi)|^2 \leq \sum_{w \in \cW_t} p_w \Big|\int e^{-2\pi i \xi f_w(x)} \, d\mu(x)\Big|^2.$$
Opening up we see that
\begin{align*}
\sum_{w \in \cW_t} p_w \Big|\int e^{-2\pi i \xi f_w(x)} \, d\mu(x)\Big|^2 &= \iint \sum_{w \in \cW_t} p_w  e^{-2\pi i \xi (f_w(x)-f_w(y))} \, d\mu(x) \, d\mu(y). 
\end{align*}
\end{proof}

Thus to prove Fourier decay, we would need to prove
\begin{align}\label{eq:needed}\iint \sum_{w \in \cW_t} p_w  e^{-2\pi i \xi (f_w(x)-f_w(y))} \, d\mu(x) \, d\mu(y) \to 0\end{align}
as $|\xi| \to \infty$ for a suitable $t = t(\xi) \to \infty$. If we want a rate for the Fourier decay, we need to control the speed of convergence in \eqref{eq:needed}. In order to do this, we first write $\xi$ as follows
$$\xi=se^t$$
using parameters $s\in\R$ and $t>0$. Later we will first take $|s|$ large, then take $t > 0$ large enough depending on $s$. Using these parameters, write
\begin{align}\label{eq:delta}\delta := |s|^{-1/2}.\end{align}
Then define the tube in $\R^2$:
$$A_\delta := \{(x,y) \in \R^2 : |x-y| \leq \delta\}.$$
We will split \eqref{eq:needed} into two cases depending on how close $x$ and $y$ are in terms of the $\delta > 0$ defined above. We will have the following two propositions given in Proposition \ref{prop:nearby} and Proposition \ref{prop:far}, which together imply Theorem \ref{thm:main}. For the quantitative part, we also need Proposition \ref{prop:farquant} to control the rate in \eqref{eq:needed}.

\subsection{Controlling nearby points} \label{sec:HY}

The first one is on the nearby points $x,y \in \R$, that is, those with $|x-y| \leq \delta$, and here is where we use the fact that $F$ is not a singleton. 
By \cite[Proposition 2.2]{FengLau}, due to $F$ is not a singleton, there exist $r_0 > 0$, $\alpha > 0$ and $C > 0$ such that for all $0 < r < r_0$ and $x\in F$ we have
\begin{align}\label{eq:decay}
\mu(B(x,r)) \leq Cr^{\alpha}.
\end{align}
A measure which satisfies this condition is sometimes called \textit{H\"older regular}. Using the decay \eqref{eq:decay} of the $\mu$ measure on balls, we can control the nearby points in the following lemma:

\begin{prop}\label{prop:nearby} There exists $C > 0$ such that for any $0 < \delta < r_0$, where $r_0 > 0$ is chosen such that the condition \eqref{eq:decay} hold for $\mu$, and for all $\xi \in \R$, we have
$$\Big|\iint_{A_\delta} \sum_{w \in \cW_t} p_w  e^{-2\pi i \xi (f_w(x)-f_w(y))} \, d\mu(x) \, d\mu(y) \Big| \leq C \delta^\alpha.$$
\end{prop}

\begin{proof} First of all, since for all $t > 0$ we have that
$$\sum_{w \in \cW_t} p_w = 1.$$
Thus we can bound using the triangle inequality as follows
$$\Big|\iint_{A_\delta} \sum_{w \in \cW_t} p_w  e^{-2\pi i \xi (f_w(x)-f_w(y))} \, d\mu(x) \, d\mu(y)\Big| \leq (\mu \times \mu) (A_\delta)$$
since $|e^{i\theta}| = 1$ for all $\theta \in \R$. Using Fubini's theorem we see that
\begin{align}\label{eq:fubini}(\mu \times \mu) (A_\delta) = \int \mu(B(x,\delta)) \, d\mu(x),\end{align}
thus by \eqref{eq:decay} the right-hand side is bounded by a constant multiple of $\delta^\alpha$.
\end{proof}

\subsection{Application of the renewal theorem and high-oscillations}

In the case when $x,y \in \R$ are chosen such that $|x-y| > \delta$, we will use the renewal theory to prove the following convergence:

\begin{prop}\label{prop:far} Suppose $\log r_j / \log r_\ell$ is irrational for some $j \neq \ell$. If $\xi = s e^t$ and $\delta = |s|^{-1/2}$, then
$$\lim_{|s|\to\infty}\lim_{t \to \infty}  \iint_{\R^2 \setminus A_\delta} \sum_{w \in \cW_t} p_w  e^{-2\pi i \xi (f_w(x)-f_w(y))} \, d\mu(x) \, d\mu(y) = 0.$$
\end{prop}	

The proof of Proposition \ref{prop:far} follows from the renewal theorem for stopping time of random walks (Proposition \ref{prop:renewal}). The rate in Proposition \ref{prop:far} is not quantitative. In the later section, by adding an extra assumption ($\log r_j / \log r_\ell$ is diophantine) to the renewal theory, we can apply the quantitative version (Proposition \ref{prop:farquant}). 

\begin{proof}[Proof of Proposition \ref{prop:far}]
By definition of $f_w$ we have that for all $x,y \in [0,1]$ and $w \in \cW_t$ the difference
$$f_w(x)-f_w(y) = r_w (x-y).$$
Therefore we can write
$$e^{-2\pi i \xi (f_w(x)-f_w(y))} = e^{-2\pi i \xi (x-y) r_w}$$
Recall that we have fixed $s \in \R$ and $t > 0$ such that $\xi$ has the form $\xi = se^{t}$. With this $s \in \R$, we can define a smooth function $g_{s} : \R \to \C$ by
$$g_s(r) := \exp(-2\pi i s e^{-r}), \quad r \in \R.$$
Then the local $C^1$ norm of $g_s$ satisfies $\|g_s\|_{C^1}=O(|s|)$, as $|s| \to \infty$, recall \eqref{eq:localC1} for the definition of local $C^1$ norm. Using $g_s$ and \eqref{equ:et} we can write for any pair $x,y \in \R$ that
\begin{equation}\label{eq:identityren}\sum_{w \in \cW_t}  p_w e^{-2\pi i \xi (f_w(x)-f_w(y))} = \E(g_{s(x-y)}(S_{n_t} - t)).\end{equation}

Due to the irrationality of $\log r_j / \log r_\ell$ for some $j \neq \ell$, the additive subgroup generated by $-\log r_j$, $j \in \cA$, is dense in $\R$. So $\lambda$ is non-lattice and we can apply Proposition \ref{prop:renewal}. 
We apply Proposition \ref{prop:renewal} with the function $g = g_{s(x-y)}$, which gives us for $\sigma = \int x \, d\lambda(x)$ that
$$\lim_{t \to \infty} \E(g_{s(x-y)}(S_{n_t} - t)) = \frac{1}{\sigma} \int_\R g_{s(x-y)}(r) p(r) \, dr.$$
If we now look at the right-hand side, since $p(r)$ is a piecewise continuous function, or just integrable, Riemann-Lebesgue lemma implies that for all $x,y \in \R^2 \setminus A_\delta$ we have
\begin{align}\label{eq:lim}\lim_{|s| \to \infty} \int_\R g_{s(x-y)}(r) p(r) \, dr = 0.\end{align}

However, to be able to use the above convergence \eqref{eq:lim}, we need uniformity in terms of $x$ and $y$ in this convergence and to make it more effective using the error term in the renewal theorem Proposition \ref{prop:renewal}. 

Recall that in \eqref{eq:delta} we defined $\delta > 0$ depending on $s \in \R$ as $\delta = |s|^{-1/2}$ and $A_\delta = \{(x,y) \in \R^2 : |x-y| > \delta\}$. Thus if $s \in \R$ and $(x,y) \in (\supp \mu)^2 \setminus A_\delta$, we have
\begin{equation}\label{equ:sxy}
|s(x-y)|\in[|s|^{1/2},C|s|],
\end{equation}
where $C > 0$ is the diameter of $\supp \mu$. 

Therefore, by \eqref{eq:identityren} and applying Proposition \ref{prop:renewal}
\begin{align*}
	&\lim_{t \to \infty}  \iint_{\R^2 \setminus A_\delta} \sum_{w \in \cW_t} p_w  e^{-2\pi i \xi (f_w(x)-f_w(y))} \, d\mu(x) \, d\mu(y) \\
	=\,&\lim_{t \to \infty}  \iint_{\R^2 \setminus A_\delta} \E(g_{s(x-y)}(S_{n_t}-t))\, d\mu(x) \, d\mu(y)\\
	=\, &\iint_{\R^2 \setminus A_\delta} \int_\R g_{s(x-y)}(r)p(r)\, dr\, d\mu(x) \, d\mu(y) , 
\end{align*}
where we can take the limit because the error term $$o_t\|g_{s(x-y)}\|_{C_1}\leq o_t\sup_{s_0\in [|s|^{1/2},C|s|]}\|g_{s_0}\|_{C_1}$$
is uniform for $(x,y)\in(\supp \mu)^2\backslash A_\delta$ by \eqref{equ:sxy}. The proof is complete by \eqref{eq:lim}.



\end{proof}

Proposition \ref{prop:far} together with Proposition \ref{prop:nearby} and Lemma \ref{lma:estimatefourier} completes the proof of Theorem \ref{thm:main} as follows:

\begin{proof}[Proof of Theorem \ref{thm:main}] By Proposition \ref{prop:nearby}, we have for all large enough $|s|$ with $\delta = |s|^{-1/2} < r_0$ and for all $t > 0$ that:
$$\Big| \iint_{A_\delta} \sum_{w \in \cW_t} p_w  e^{-2\pi i \xi (f_w(x)-f_w(y))} \, d\mu(x) \, d\mu(y)\Big| \leq C|s|^{-\alpha/2},$$
where $\xi = s e^t$, $\alpha > 0$ and $C > 0$ is a universal constant. Thus
$$\limsup_{t\to\infty}\Big| \iint_{A_\delta} \sum_{w \in \cW_t} p_w  e^{-2\pi i \xi (f_w(x)-f_w(y))} \, d\mu(x) \, d\mu(y)\Big| \leq C|s|^{-\alpha/2}.$$
On the other hand, Proposition \ref{prop:far} implies that we have
$$\lim_{|s|\to\infty}\lim_{t\to\infty}\Big| \iint_{\R^2 \setminus A_\delta} \sum_{w \in \cW_t} p_w  e^{-2\pi i \xi (f_w(x)-f_w(y))} \, d\mu(x) \, d\mu(y)\Big|= 0,$$
Hence Lemma \ref{lma:estimatefourier} gives
$$\lim_{|s| \to \infty}\limsup_{t\to\infty} |\widehat{\mu}(se^t)|^2 = 0,$$
which gives the claim.
\end{proof}

\subsection{Quantitative rate for Fourier decay}

In order to prove a quantitative rate (Theorem \ref{thm:mainquantitative}), we need a rate in Proposition \ref{prop:far}. For this purpose we will employ Proposition \ref{prop:renewalquantitative}. So we need to verify weakly diophantine condition for $\lambda$, which will follow from the diophantine assumption in Theorem \ref{thm:mainquantitative}. From now on, we use a notation $A \ll B$ to mean there exists a universal constant $C > 0$ such that $A \leq C B$, and for $A \gg B$ that $A \geq C B$ respectively.

\begin{lemma}\label{lma:dioph}If there exist $r_j, r_k$ for $j,k\in\AA$ such that $\log r_j/\log r_k$ is diophantine, then the measure $\lambda$ is weakly diophantine.
\end{lemma}
\begin{proof}
 Indeed, we have
  \begin{align*}
  |1-\cL\lambda(ib)| &\geq |\Rea(p_j(1-e^{-ib\log r_j})+p_k(1-e^{-ib\log r_k}))|\\
  &\gg \max\{d(b\log r_j ,2\pi\Z)^2, d(b\log r_k,2\pi \Z)^2 \}\\
  & \gg \max\Big\{d(b_1,\Z)^2,d\Big(b_1\frac{\log r_k}{\log r_j},\Z\Big)^2 \Big\},
  \end{align*}
  with $b_1=b\log r_j/2\pi$. By the definition \eqref{eq:diophantine} of a diophantine number, we obtain that for some $l>0$
  $$\max\Big\{d(b_1,\Z)^2,d\Big(b_1\frac{\log r_k}{\log r_j},\Z\Big)^2\Big\}\gg |b_1|^{-2l},$$
  which implies $\lambda$ is weakly diophantine.
  \end{proof}

Now we can prove the following quantitative version of Proposition \ref{prop:far}, which implies Theorem \ref{thm:mainquantitative}:

\begin{prop}\label{prop:farquant} Suppose $\log r_j / \log r_\ell$ is diophantine for some $j \neq \ell$. Then there exists $\beta>0$ such that
$$\Big|\iint_{\R^2 \setminus A_\delta} \sum_{w \in \cW_t} p_w  e^{-2\pi i \xi (f_w(x)-f_w(y))} \, d\mu(x) \, d\mu(y)\Big| = O\Big(\frac{1}{|\log |\xi||^{\beta/4}}\Big),$$
as $|\xi| \to \infty$, where $t = t(\xi)$ satisfies $|\xi|=t^{\beta/2}e^t$.
\end{prop}

\begin{proof} Since $\lambda$ is weakly diophantine by Lemma \ref{lma:dioph}, we can apply Proposition \ref{prop:renewalquantitative} to obtain for some $\beta > 0$ that for $\sigma = \int x \, d\lambda(x)$:
$$\Big|\E(g_{s(x-y)}(S_{n_t} - t)) - \frac{1}{\sigma}\int_\R g_{s(x-y)}(r) p(r) \, dr\Big| = O\Big(\frac{|s(x-y)|}{t^\beta}\Big).$$
Because the function $p(r)=\int_{x>r}d\lambda(x)$, $r\geq 0$, is piecewise constant with a finite number of points of discontinuity, the decay rate in the main term, the oscillation integral, is given by the oscillation (See \cite[Lemma 3.8]{Li1} for more details)
$$\Big|\int_{\R} g_{s(x-y)}(r) p(r)\, dr \Big| =O\Big(\frac{1}{|s(x-y)|}\Big) .$$
Then we take $|s| = t^{\beta/2}$, which implies $|s(x-y)|\in [t^{\beta/4},Ct^{\beta/2}]$ for $(x,y)\in (\supp \mu)^2 \backslash A_\delta$ for $C$ equal to the diameter of the support of $\mu$. When $|\xi|=t^{\beta/2}e^t$, we have, after taking logarithms that the rate is given by $O(|\log|\xi||^{-\beta/4})$, which gives the claim.
\end{proof}

\section{Proofs of the renewal theorems}
\label{secrentheory}

\def\linf{\infty}
\def\Res{E}
\def\Cut{E_C}
\def\lf{{C^1_x}}
\def\ck{{Lip}}
\def\Lip{{\infty}}
\def\pexp{B_s}
\def\pexpo{B}
\def\cal{\mathcal}
\newcommand{\dd}{\, d}


In this section we will prove Proposition \ref{prop:renewal} and \ref{prop:renewalquantitative}. They follow the similar proofs for the stationary measure for Lie group actions in \cite{Li1}, but here we will give a full self-contained proof in order to present the method in the most basic setting. Let us now give a brief overview of the proofs presented here. First of all, we will give a proof of key renewal theorem for good functions using Laplace transform in Proposition \ref{prop:renreg}, with a control of error term. Then in order to find the limit of the distribution of the process $S_{n_t}-t$, we will consider the joint distribution of $(X_{n+1},S_n-t)$ and prove a renewal theorem for this joint distribution in Proposition \ref{prop:residue} using Proposition \ref{prop:renreg}. Then we add the cutoff assumption that $S_n-t<0$ and $X_{n+1}+(S_n-t)\geq 0$. In Proposition \ref{prop:rescut} we prove a renewal theorem for the joint distribution of $(X_{n+1},S_n-t)$ under this cutoff assumption. Finally, the reason we are doing this is that the distribution of $S_{n_t}-t$ is exactly the sum over natural numbers of the distribution of the sum $X_{n+1}+(S_n-t)$ under the cutoff assumption of $(X_{n+1},S_n-t)$.

In the following sections we actually only need that $\lambda$ is supported on $\R^+$, non-lattice, and has an \textit{exponential moment}: there exists $\eps > 0$ such that
$$\int e^{\eps x}\dd\lambda(x) < \infty.$$
The exponential moment assumption is more general than the finite support hypothesis of $\lambda$ in Section \ref{sec:renewal}. However, in the proofs of Propositions \ref{prop:renewal} and \ref{prop:renewalquantitative} we will impose that $\lambda$ has a finite support, see Section \ref{sec:proofrenewal} for their proofs. With a bit more effort, Proposition \ref{prop:renewal} and \ref{prop:renewalquantitative} can also be obtained with exponential moment, but we don't investigate this generality here. 

\subsection{Laplace transform}
The Laplace transform of a probability measure $\lambda$ with finite exponential moment on $\R$ is defined by
\begin{align*}
\cL\lambda(z)=\int e^{-z x}\dd\lambda(x),
\end{align*} 
for $\Re z>-\eps$, where $\eps>0$ is the constant in the definition of exponential moment of $\lambda$.

Recall the expectation $\sigma=\int x\dd\lambda(x)>0$. By the definition of non-lattice and $\lambda$ having exponential moment, we have

\begin{prop}\label{prop:invtran}
	If $\lambda$ is non-lattice, then for any pure imaginary number $i\xi$ not $0$, the Laplace transform of $\lambda$ is not equal to $1$ and
	\begin{equation}\label{equ:i-pz}
	u(z):=\frac{1}{1- \cL\lambda(z)}-\frac{1}{\sigma z}
	\end{equation}
	is holomorphic on a neighbourhood of the half plane $\{z\in\C,\ \Re z\geq 0 \}$.
\end{prop}

\subsection{Key Renewal theory for good functions}
We start to compute the renewal operator. In this section we will prove a result for the renewal operator for ``good" functions. Let us first fix some notations. Let $f$ be a Borel function on $\R$. Let $\|f\|_{\infty} :=\sup_{x\in\R}|f(x)|$ be the supremum norm, $\|f\|_{L^p}$ the $L^p$ norm with respect to Lebesgue measure, and $L^p(\R)$ spaces with respect to Lebesgue measure. Define a Sobolev norm 
$$\|f\|_{W^{1,\infty}}:= \|f\|_{\infty}+\|f'\|_{\infty}.$$ 
In this section, we use the following definition of Fourier transform 
$$\hat f(\xi)=\int e^{-ix\xi}f(x)\dd x$$
for functions $f : \R \to \R$ as opposed to the one in the introduction for Borel measures $\mu$ on $\R$. For a compact set $|K|$ in $\R$, we denote by $|K|$ the supremum of absolute value of elements in $K$, that is
\[|K|:=\sup\{|x| : x \in K \}. \]
Recall that we defined the renewal operator for a non-negative bounded function $f$ on $\R$ by
\begin{align*}
Rf(t) = \sum_{n=0}^{\infty}\E (f(S_n-t)) = \sum_{n=0}^{\infty} \int f(x-t) \, d\lambda^{\ast n}(x), \quad t \in \R.
\end{align*}

\begin{prop}\label{prop:renreg}
	Let $f$ be a non-negative bounded continuous function in $L^1( \R)$ such that its Fourier transform satisfies $\hat f\in W^{1,\infty}(\R)$. Assume that $\supp \hat{f}$ is in a compact interval $K$. Then for all $t>0$, we have
	\[Rf(t)=\frac{1}{\sigma}\int_{-t}^{\infty}f(x)\dd x+ \frac{1}{t}O_K\|\hat f\|_{W^{1,\infty}}, \]
	where $O_K$ satisfies
	$$O_K\leq |K|\sup\{|u(i\xi)|+|\partial_\xi u(i\xi)| : \xi\in K \}.$$
\end{prop}
The proof of Proposition \ref{prop:renreg} follows by combining the following two lemmas. Firstly, we have:

\begin{lemma}\label{lem:rencon}
	Under the same assumption as in Proposition \ref{prop:renreg}, we have
	\begin{align*}
	Rf(t)=\frac{1}{\sigma}\int_{-t}^{\infty}f(x)\dd x+\frac{1}{2\pi}\int e^{-it\xi}u(-i\xi)\hat{f}(\xi)\dd\xi.
	\end{align*}
\end{lemma}	
\begin{proof}
This is a classical computation, but for completeness, we will include a proof here.
	Introduce a local notation: for $t$ in $\R$ and $s\geq 0$, let $B_s$ be the operator defined by
	\begin{align*}
	\pexp f(t)=\int e^{-s x}f(x+t)\dd\lambda(x).
	\end{align*}
	Then for $n\in\N$,
	\begin{align*}
	\pexp^n(f)(t)&=\int e^{-sx}f(x+t)\dd\lambda^{*n}(x).
	\end{align*}
	When $s=0$, 
	we have
	\begin{equation}\label{one}
	Rf(-t)=\sum_{n\geq 0}B_0^nf(t).
	\end{equation}
	
	Since $f\geq 0$ and $x>0$ in the support of $\lambda$, using the monotone convergence theorem, we have
	\begin{align*}
	\lim_{s\rightarrow 0^+}\sum_{n\geq 0}\int e^{-sx}f(x+t)\dd\lambda^{*n}(x) = \sum_{n\geq 0} \int f(x+t)\dd\lambda^{*n}(x).
	\end{align*}
	Thus
	\begin{equation}\label{equ:slim}
	\sum_{n\geq 0} \pexpo_0^n(f)(t)=\lim_{s\rightarrow 0^+}\sum_{n\geq 0} \pexp^n(f)(t).
	\end{equation}
	
	Using the inverse Fourier transform, we have
	\begin{equation}\label{equ:bsn1}
	\sum_{n\geq 0} \pexp^n(f)(t)
	=\sum_{n\geq 0}\int e^{-sx}\frac{1}{2\pi}\int_{\R}e^{i\xi(x+t)}\hat{f}( \xi)\dd\xi\dd\lambda^{*n}(x).
	\end{equation}
	Since $\hat{f}(\xi)$ has compact support and $|\hat{f}(\xi)|$ is bounded, we know that $\|\hat{f}\|_{L^1}$ is finite. For $s>0$, by $\cL\lambda(s)<1$, we have
	\begin{align*}
	\sum_{n\geq 0}\int e^{-sx}\int_{\R}|\hat{f}( \xi)|\dd\xi\dd\lambda^{*n}(x)= \|\hat{f}\|_{L^1}\sum_{n\geq 0}\int e^{-sx}\dd\lambda^{*n}(x)=\|\hat{f}\|_{L^1}\sum_{n\geq 0}\cL\lambda(s)^n< \infty,
	\end{align*} 
	which implies that the right hand side of \eqref{equ:bsn1} is absolutely convergent.
	Consequently, we can use the Fubini theorem to change the order of the integration. By the hypothesis $\hat{f}(\xi)\in W^{1,\infty}(\R)$, Proposition \ref{prop:invtran} implies that for $s>0$
	\begin{equation}\label{equ:bsn}
	\begin{split}
	\sum_{n\geq 0} \pexp^n(f)(t)&=\frac{1}{2\pi}\int_{\R}\sum_{n\geq 0}\int  e^{(-s+i\xi)x}\hat{f}( \xi)\dd\lambda^{*n}(x)e^{it\xi}\dd\xi\\
	&=\frac{1}{2\pi}\int_{\R}\sum_{n\geq 0} \cL\lambda(s-i\xi)^n\hat{f}(\xi)e^{it\xi}\dd\xi\\
	&=\frac{1}{2\pi}\int_{\R}(1- \cL\lambda(s-i\xi))^{-1}\hat{f}(\xi)e^{it\xi}\dd\xi\\
	&=\frac{1}{2\pi}\int_{\R}\left(\frac{1}{\sigma(s-i\xi)}+u(s-i\xi)\right)\hat{f}(\xi) e^{it\xi}\dd\xi.
	\end{split}
	\end{equation}
	Since $\frac{1}{s-i\xi}=\int_{0}^{+\infty}e^{-(s-i\xi)x}\dd x$ for $s>0$, together with the property $\hat{f}\in L^1(\R)$, we have
	\begin{align}\label{two}
	\frac{1}{2\pi}\int_{\R}\frac{1}{\sigma(s-i\xi)}\hat{f}(\xi) e^{it\xi}\dd\xi
	=\frac{1}{\sigma}\int_{0}^{\infty}f(x+t)e^{-sx}\dd x.
	\end{align}
	When $s\rightarrow 0^+$, since $f$ is integrable, by monotone convergence theorem, the limit is $\frac{1}{\sigma}\int_{t}^{\infty}f(x)\dd x$.
	Since $\hat f(\xi)$ is compactly supported, we have
	\begin{align}\label{three}
	\lim_{s\rightarrow 0^+}\int_{\R}u(s-i\xi)\hat{f}(\xi) e^{it\xi}\dd\xi=\int_{\R}u(-i\xi)\hat{f}(\xi) e^{it\xi}\dd\xi.
	\end{align}
	The proof is complete by combining \eqref{one}-\eqref{three}.
\end{proof}

\begin{lemma}\label{lem:renres} 
	Under the same assumption as in Proposition \ref{prop:renreg}, we have
	\[\Big|\int e^{-it\xi}u(-i\xi)\hat{f}(\xi)\dd\xi\Big|\leq \frac{1}{t}O_K\left(\|\hat f\|_{\infty }+\|\partial_\xi\hat f\|_{\infty }\right), \]
	where $O_K$ is from Proposition \ref{prop:renreg}.
\end{lemma}
\begin{proof}
	Use the fact that $\hat f(\xi)$ is compactly supported and $|\hat{f}(\xi)|_{ },\,|\partial_{\xi}\hat{f}(\xi)|_{ }<\infty$. Then applying integration by parts, we have
	\begin{align*}
	\int e^{-it\xi}u(-i\xi)\hat{f}(\xi)\dd\xi&=\frac{1}{it}\int e^{-it\xi}\partial_{\xi}(u(-i\xi)\hat{f}(\xi))\dd\xi\\
	&=\frac{1}{it}\int e^{-it\xi}\left(\partial_{\xi}(u(-i\xi))\hat{f}(\xi)+u(-i\xi)\partial_{\xi}\hat{f}(\xi)\right)\dd\xi.
	\end{align*}
	Since $|u(i\xi)|$ and $|\partial_\xi u(i\xi)|$ are uniformly bounded on compact regions, the result follows.
\end{proof}

\subsection{Regularity properties of renewal measures}\label{sec:regular}

We want to use convolution to smooth out the target function. There exists a non-negative even function $\psi$ such that it is a probability density, and the Fourier transform $\hat{\psi}$ is compactly supported on $[-1,1]$. For example we can take $\hat{\psi}=\tau*\tau$ where $\tau$ is a smooth even function supported on $[-1/2,1/2]$ and $\psi$ is determined via inverse Fourier transform.
Write $\psi_{\delta}(t) := \frac{1}{\delta^2}\psi(\frac{t}{\delta^2})$. Since $\psi$ decays faster than any polynomial, there exists $C_1>0$ such that 
$$\int_{-\delta}^{\delta}\psi_{\delta}(t)\dd t=\int_{-1/\delta}^{1/\delta}\psi(t)\dd t>1-C_1\delta.$$
Then we have the following approximation theorem for the renewal operators of indicator functions:
\begin{prop}\label{prop:renint}
	Let $\delta\leq 1/3C_1$ and $b\geq a$. If $b-a\geq 2\delta$, then for $t>0$, we have
	\begin{equation}\label{ineq:renint}
	R(\one_{[a,b]})(t)\leq 3(b-a)(1/\sigma+C_\psi O_{\delta}(1+|b|+|a|)/t),
	\end{equation}
	where $O_\delta :=O_{[-\delta^{-2},\delta^{-2}]}$ and $C_\psi=1+ \|x \mapsto x\psi(x)\|_{L^1}$. Here $O_{[-\delta^{-2},\delta^{-2}]}$ is from Proposition \ref{prop:renreg} with $K = [-\delta^{-2},\delta^{-2}]$.
\end{prop}
\begin{proof}
	If $x$ is in $[a,b]$, then $[x-b,x-a]$ contains at least one of $[0,\delta]$ or $[-\delta,0]$. Therefore
	\[\psi_{\delta}*\one_{[a,b]}(x)=\int_{a}^{b}\psi_{\delta}(x-v)\dd v\geq \int_{0}^{\delta}\psi(v)\dd v\geq (1-C_1\delta)/2. \]
	Then by $C_1\delta\leq 1/3$,
	\begin{equation}\label{ineq:psidellarg}
	\one_{[a,b]}\leq 3\psi_{\delta}*\one_{[a,b]}.
	\end{equation}
	It is sufficient to bound $R(\psi_{\delta}*\one_{[a,b]})(t)$. Proposition \ref{prop:renreg} implies that
	\[R(\psi_{\delta}*\one_{[a,b]})(t) = \frac{1}{\sigma}\int_{-t}^{\infty}\psi_{\delta}*\one_{[a,b]}(x) \, dx+\frac{O_\delta}{t}\|\hat{\psi_\delta}\hat{\one}_{[a,b]}\|_{W^{1,\infty} }. \]
	The first term is less than $\int\psi_{\delta}*\one_{[a,b]}=(b-a)$. For the second term, we have
	\begin{align*}
	\|\hat{\psi_\delta}\hat{\one}_{[a,b]}\|_{W^{1,\infty} }&=\|\hat{\psi_\delta}\hat{\one}_{[a,b]}\|_{\infty }+\|\partial_\xi(\hat{\psi_\delta}\hat{\one}_{[a,b]})\|_{\infty }\\
	&\leq (1+\|x \mapsto x\psi_\delta(x)\|_{L^1})(\|\one_{[a,b]}\|_{L^1}+\ \|x \mapsto x\one_{[a,b]}(x)\|_{L^1})\\
	& \leq C_\psi(b-a)(1+|a|+|b|).
	\end{align*}
\end{proof}

Because every step of the random walk $X_1,X_2,\dots$ is positive, every trajectory can only stay at most $Cs$ times for $s\geq 1$ in the interval $[t,t+s]$, with $C$ depending on $\lambda$. Recall we use a notation $A \ll B$ to mean there exists a universal constant $C > 0$ such that $A \leq C B$.
\begin{lemma}\label{lem:renintts}
	For all $s \geq 1$ and $t \in \R$, we have
	\begin{equation}\label{ineq:renintts}
	R(\one_{[0,s]})(t)\ll \max\{1,s \}.
	\end{equation}
\end{lemma}

\subsection{Residue process}
\label{subsecrespro}
We introduce the \textit{residue process}, which not only deals with $S_n$ but also takes into account the next step $X_{n+1}$. 
Let $f$ be a non-negative bounded Borel function on $ \R^2$. For $t\in  \R$, we define the \textit{residue operator} by
\begin{equation}
\begin{split}
\Res f(t)&:=\sum_{n\geq 0}\iint f(y, x-t)\dd\lambda^{*n}(x)\dd\lambda(y).
\end{split}
\end{equation}
For clarity, we will from now on use the notations $\R_x$, $\R_y$ and $\R_\xi$ to be real lines but highlighted the coordinate we use. Here the space coordinates are $x$ and $y$ and frequency (Fourier) coordinates are denoted by $\xi$.

Let
$$\calF_xf(y,\xi) := \int f(y,x)e^{-ix\xi}\dd x$$ 
be the Fourier transform of $f$ on $\R_x$. Let $F$ be a function on $  \R_y\times\R_{\xi}$,. Define the infinity norm by
\[\|F\|_\Lip=\sup_{y,\xi\in\R}|F(y,\xi)|.  \]

\begin{prop}[Residue process]\label{prop:residue}
	Let $f$ be a non-negative bounded continuous function on $ \R^2$. Assume that the projection of $\supp\calF_x(f)$ onto $\R_\xi$ is contained in a compact interval $K$, and $\|\calF_x(f)\|_\Lip,\|\partial_\xi\calF_x(f)\|_\Lip$ are finite. 
	Then for $t>0$, we have 
	\begin{equation}
	\begin{split}
	\Res f(t)=&\frac{1}{\sigma}\int_{-t}^\infty\int_{\R^+} f(y,x)\dd\lambda(y)\dd x+\frac{1}{t}O_K\left(\|\calF_x(f)\|_\Lip+\|\partial_\xi\calF_x(f)\|_\Lip \right),
	\end{split}
	\end{equation}
	where $O_K$ is from Proposition \ref{prop:renreg}.
\end{prop}
\begin{proof}
	For a bounded continuous function $f$ on $  \R^2$ and $x\in  \R$, we define an operator $Q$ by
	\[Qf(x)=\int f(y,x)\dd\lambda(y). \]
	Then
	\[\Res f(t)=\sum_{n\geq 0}\int Qf( x-t)\dd\lambda^{*n}(x)=R(Qf)(t). \]
	
	We want to use Proposition \ref{prop:renreg}, so we need to verify the hypotheses. The function $Qf$ is bounded since $f$ is bounded and integrable since $\|\calF_x f\|_{\infty}$ is finite. Then
	\begin{align*}
	\widehat{Qf}(\xi)&=\int Qf(x)e^{-ix\xi}\dd x=\int f(y,x)e^{-ix\xi}\dd x\dd\lambda(y)=\int\calF_xf(y,\xi)\dd\lambda(y).
	\end{align*}
	Thus $\widehat{Qf}$ is also compactly supported on $\R_\xi$. 
	\begin{lemma}[Change of norm]
		Under the assumptions of Proposition \ref{prop:residue}, we have\[\|\widehat{Qf}\|_{\linf}\leq\|\calF_x(f)\|_\Lip,\  \|\partial_\xi\widehat{Qf}\|_{\linf}\leq\|\partial_\xi\calF_xf\|_\Lip.\]
	\end{lemma}
	\begin{proof}
		The second inequality follows by the same computation as $\widehat{Qf}$.
	\end{proof}
	By Proposition \ref{prop:renreg}, we have
	\begin{align*}
	R(Qf)(t&)=\frac{1}{\sigma} \int_{-t}^{\infty}Qf(x)\dd x+\frac{1}{t}O_K\left(\|\widehat{Qf}\|_{\linf}+\|\partial_\xi \widehat{Qf}\|_{\linf}\right)\\
	&=\frac{1}{\sigma} \int_{-t}^{\infty}Qf(x)\dd x+\frac{1}{t}O_K\left(\|\calF_x(f)\|_\Lip+\|\partial_\xi \calF_x(f)\|_\Lip\right).
	\end{align*}
	The proof of Proposition \ref{prop:residue} is complete.
\end{proof}

\subsection{Residue process with cutoff}
In this section, we restrict the residue process to the sequences $(X_1,\dots,X_n, X_{n+1})$ such that $S_n<t\leq S_n+ X_{n+1}$. For a $C^1$ function $f$ on $\R_y\times \R_x$, define an $x$-coordinate partial derivative $C^1$ norm by
\begin{align}
\| f \|_{\lf} := \| f\|_\infty+\|\partial_xf\|_\infty.
\end{align}
Define a \textit{cutoff operator} $\Cut$ from non-negative Borel functions on $ \R^2$ to functions on $ \R$ by
\[\Cut f(t) := \sum_{n\geq 0}\iint_{ x<t\leq y+x} f(y, x-t)\dd\lambda(y)\dd\lambda^{*n}(x). \]
Then we have:

\begin{lemma}\label{lem:resfin}
	There exists $C_2>0$ such that for all $t\in\R$, we have
	\begin{equation}\label{ineq:resfin}
	\Cut(\one)(t)=\Res(\one_{-y\leq x<0})(t)\leq C_2.
	\end{equation}
\end{lemma}
\begin{proof}
	By Lemma \ref{lem:renintts}, we have
	\begin{align*}
	\quad\sum_{n\geq 0}\lambda\otimes\lambda^{*n}\{(y,x)| x-t\in[- y,0], y\geq 0\}=\int R(\one_{[-y,0]})(t)\dd\lambda(y)\ll \int\max\{1,y\}\dd\lambda(y),
	\end{align*}
	which is the claim by the definitions of $E_C$ and $E$.
\end{proof}

By Lemma \ref{lem:resfin}, this cutoff operator $\Cut$ is actually well defined for bounded Borel functions. 
\begin{prop}\label{prop:rescut}
	Let $f$ be a continuous function on $ \R^2$ with $\| f\|_\lf$ finite. Assume that the projection of $\supp f$ on $\R_y$ is contained in a compact set $K$.
	For all $1/3C_1>\delta>0$ and $t>|K|+\delta$, we have
	\begin{equation}
	\begin{split}
	\Cut f(t)=\int_{\R^+}\int_{-y}^{0} f(y,x)\dd x\dd\lambda(y) +O_K(\delta +O_\delta/t)\| f\|_\lf,
	\end{split}
	\end{equation}
	where $O_K$ only depends on $K$ and $\lambda$, and
	$$O_\delta := \sup\{|u(i\xi)|+|\partial_\xi u(i\xi)| : \xi\in[-\delta^{-2},\delta^{-2}]\}.$$
\end{prop}
\begin{remark}
	We decompose $f$ into real and imaginary parts, then decompose these two parts into positive and negative parts. Each part satisfies the hypotheses of Proposition \ref{prop:rescut}, with the support and the Lipschitz norm bounded by the original one. Thus, it is sufficient to prove this proposition for $f$ non-negative.
\end{remark}
The following lemma connects the cutoff operator $E_C$ with the residue operator $E$.
\begin{lemma}
	Under the assumptions of Proposition \ref{prop:rescut}, let
	$$ f_o(y,x) :=\one_{-y\leq x<0} f(y,x).$$ 
	Then \[\Cut f(t)=\Res f_o(t).\]
\end{lemma}

Using $\psi_\delta$ to regularize these functions, we write 
$$ f_\delta(y,x):=\int  f_o(y,x-x_1)\psi_\delta(x_1)\dd x_1=\psi_\delta*f_o(y,x).$$
\begin{lemma}\label{lem:rescut}
	Under the same hypotheses as in Proposition \ref{prop:rescut}, we have
	\begin{align*}
	\Res( f_\delta)(t)=\int_{\R^+}\int_{-y}^{0} f(y,x)\dd x\dd\lambda(y) +O\Big(\delta+\frac{O_\delta}{t}(|K|+|K|^2)\Big)\| f\|_\infty.
	\end{align*}
\end{lemma}
\begin{proof}
	We want to verify the conditions in Proposition \ref{prop:residue} and then use Proposition \ref{prop:residue}. For the Fourier transform, we have
	\begin{align*}
	\calF_x f_\delta=\calF_x(\psi_\delta*f_o)=\hat{\psi}_\delta\calF_x f_o.
	\end{align*} 
	We need to estimate the infinity norm of $\calF_x f_o(y,\xi)$. This function equals
	\[\int  f_o(y,x)e^{-ix\xi}\dd x=\int_{-y}^{0} f(y,x)e^{-ix\xi}\dd x. \] 
	\begin{lemma}[Change of norm] Under the same hypotheses as in Proposition \ref{prop:rescut}, we have  
		\[\|\calF_x f_\delta\|_\Lip\leq |K|\|f\|_\infty,\ \|\partial_\xi\calF_x f_{\delta}\|_\Lip\leq |K|^2\|f\|_\infty.\]
	\end{lemma}
	\begin{proof}
		Noting that in the integration $|x|\leq |y|$, we get the second inequality by the same computation.
	\end{proof}
	The projection of the support of $\calF_xf_\delta$ onto $\R_\xi$ is contained in $[-\delta^{-2},\delta^{-2}]$. Therefore by Proposition \ref{prop:residue}, we have
	\begin{align*}
	\Res( f_\delta)(t)=\frac{1}{\sigma}\int_{-t}^\infty\int_{\R^+}  f_\delta( y,x) \dd\lambda(y)\dd x+\frac{O_\delta}{t}\left(\|f\|_\infty(|K|+|K|^2)\right).
	\end{align*}
	Then
	\begin{align*}
	\int_{-t}^{\infty} f_{\delta} (y,x)\dd x&=\int_{-t}^{\infty}\int_{-y}^{0} f(y,x_1)\psi_\delta(x-x_1)\dd x_1\dd x\\
	&=\int_{-y}^{0} f(y,x_1)\int_{-t}^{\infty}\psi_{\delta}(x-x_1)\dd x\dd x_1\\
	&=\int_{-y}^{0} f (y,x_1)\dd x_1-\int_{-y}^{0} f (y,x_1)\int_{-\infty}^{-t-x_1}\psi_{\delta}(x)\dd x\dd x_1.
	\end{align*}
	Since $t-\delta\geq |K|$, we have $-t-x_1\leq -t+y\leq -\delta$. By $\int_{-\infty}^{-\delta}\psi_\delta\leq C_1\delta$, this implies that $\int_{-t}^{\infty} f_{\delta} (y,x)\dd x=\int_{-y}^{0} f_{\delta} (y,x)\dd x(1+O(\delta))$.
	Using Lemma \ref{lem:resfin}, we have
	\[\Big|\int_{\R^+}\int_{-y}^{0} f(y,x)\dd x\dd\lambda(y) \Big|\leq \| f\|_\infty \Cut(\one)=O(\| f\|_\infty). \]
	Therefore
	\begin{align*}
	&\int_{-t}^\infty\int_{\R^+} f_\delta( y,x) \dd\lambda(y)\dd x=\int_{\R^+}\int_{-y}^{0} f( y,x)\dd x\dd\lambda(y) +O(\delta\|f\|_\infty).
	\end{align*}
	The proof of Lemma \ref{lem:rescut} is complete.
\end{proof}
Next, we will need a lemma to estimate $|f_\delta-f_o|$.

\begin{lemma}\label{lem:chadif} 
	Let $\varphi$ be a $C^1$ function with $\|\varphi'\|_{\infty}< \infty$ and $\|\varphi\|_{\infty}\leq 1$. Let $\varphi_o(u)=\one_{[a,b]}(u)\varphi(u)$ where $b>a$. Then we have
	\begin{equation}\label{ineq:chadif}
	|\psi_{\delta}*\varphi_o(u)-\varphi_o(u)|\leq
	\begin{cases}
	(\|\varphi'\|_\infty+2C_1)\delta, & u\in[a+\delta,b-\delta],\\
	2, & u\in[a-\delta,a+\delta]\cup[b-\delta,b+\delta],\\
	\psi_{\delta}*\one_{[a,b]}(u), & u\in [a-\delta,b+\delta]^c.
	\end{cases}
	\end{equation}
	If $b-a\leq 2\delta$, then $[a+\delta,b-\delta]$ is empty and we don't have the first one. 
\end{lemma}
\begin{proof} We will prove this inequality in each interval.
	\begin{itemize}
		\item When $u$ is in $[a+\delta,b-\delta]$, we have
		\begin{align*}
		|(\psi_{\delta}*\varphi_o-\varphi_o)(u)|&=\left|\int\psi_{\delta}(t)(\varphi_o(u-t)-\varphi_o(u))\dd t\right|\\
		&\leq\int_{-\delta}^{\delta}\psi_{\delta}(t)|\varphi_o(u-t)-\varphi_o(u)|\dd t+2C_1\delta.
		\end{align*}
		When $|t|\leq \delta$, we have $u-t\in[a,b]$. Since $|\varphi_o'(u)|\leq\|\varphi'\|_\infty$ for $u\in[a,b]$, this implies that
		\[\int_{-\delta}^{\delta}\psi_{\delta}(t)|\varphi_o(u-t)-\varphi_o(u)|\dd t
		\leq\int_{-\delta}^{\delta}\psi_{\delta}(t)|t|\|\varphi'\|_\infty\dd t\leq \delta\|\varphi'\|_\infty. \]
		\item When $u\in[a-\delta,a+\delta]\cup[b-\delta,b+\delta]$, we use the trivial bound $|\psi_{\delta}*\varphi_o(u)-\varphi_o(u)|\leq 2$.
		\item When $u\in(-\infty,a-\delta]\cup[b+\delta,\infty)$, we have $\varphi_o(u)=0$, then $|\psi_{\delta}*\varphi_o|\leq|\psi_{\delta}*\one_{[a,b]}|$.
	\end{itemize}
	Thus collecting all together, we get the inequality \eqref{ineq:chadif}.
\end{proof}
\begin{proof}[Proof of Proposition \ref{prop:rescut}]
	To simplify the notation, we normalize $ f$ in such a way that $\| f\|_\infty=1$. By Lemma \ref{lem:rescut}, we only need to give an estimate of $\Res(| f_\delta- f_o|)(t)$. 
	
	Due to $ f_o(y,x)=\one_{-y\leq x<0}(x) f(y,x)$, by Lemma \ref{lem:chadif}: (For clarity, we omit the variable $y$ in the following computation)
	\begin{equation*}
	| f_\delta- f_o|(x)\leq
	\begin{cases}
	(\|\partial_x f\|_\infty+2C_1)\delta, & x\in[-y+\delta,-\delta],\\
	2, & x\in[-y-\delta,-y+\delta]\cup[-\delta,\delta],\\
	\psi_\delta*\one_{[-y,0]}(x), & x\in[-y-\delta,\delta]^c.
	\end{cases}	
	\end{equation*}
	By the definition of $|K|$, the first term is less than $(|\partial_x f|_\infty+2C_1)\delta\one_{[-|K|+\delta,-\delta]}$. The third term equals to
	\begin{align*}
	\one_{[-\infty,-y-\delta]\cup[\delta,\infty]}\psi_\delta*\one_{[-y,0]}(x)&=\one_{[-\infty,-y-\delta]\cup[\delta,\infty]}(x)\int_{-y}^0\psi_\delta(x-x_1)\dd x_1\\
	&=\one_{[-\infty,-y-\delta]\cup[\delta,\infty]}(x)\int_{x}^{x+y} \psi_\delta(x_1)\dd x_1.
	\end{align*}
	
	By the definition and the above arguments, we have
	\begin{align*}
	\Res(| f_\delta- f_o|)(t)&=\sum_{n\geq 0}\int | f_\delta- f_o|(y, x-t)\dd\lambda^{*n}(x)\dd\lambda(y)\\
	& \leq\sum_{n\geq 0}\int\Big((\|\partial_x f\|_\infty+2C_1)\delta\,\one_{[-|K|,-\delta]}( x-t) +2\,\one_{[- y-\delta,- y+\delta]\cup[-\delta,\delta]}( x-t)\\
	&\qquad\qquad\quad +\one_{[-\infty,- y-\delta]\cup[\delta,\infty]}( x-t)\int_{ x-t}^{x+y-t}\psi_\delta(x_1)\dd x_1\Big) \dd\lambda^{*n}(x)\dd\lambda(y).
	\end{align*}
	By Lemma \ref{lem:renintts}, the first term is controlled by $(\|\partial_x f\|_\infty+2C_1)\delta |K|$. For the second term,
	we have
	\[\int \one_{[- y-\delta,- y+\delta]}( x-t)\dd\lambda^{*n}(x)\dd\lambda(y)=\int \one_{[-\delta,\delta]}(x-t)\dd\lambda^{*(n+1)}(x). \]
	 So the second term is less than $4R(\one_{[-\delta,\delta]})(t)$. Due to Proposition \ref{prop:renint}, it is controlled by $C_\psi\delta(1/\sigma+O_\delta(1+2\delta)/t)$. 
	
	For the third term, we need to change the order of integration. Since $ x-t>\delta$ or $ x-t<- y-\delta$, we have $x_1\geq x-t>\delta$ or $x_1\leq x+y-t\leq -\delta$. We first integrate with respect to $x_1$, and so the third term is less than
	\begin{align*}
	&\int_{[-\infty,-\delta]\cup[\delta,\infty]}\psi_\delta(x_1)\sum_{n\geq 0}\lambda\otimes\lambda^{*n}\{(y,x)|x+y\geq x_1+t\geq x \}\dd x_1\\
	&=\int_{[-\infty,-\delta]\cup[\delta,\infty]}\psi_\delta(x_1)E_C(\one)(x_1+t)\dd x_1.
	\end{align*}
	By Lemma \ref{lem:resfin}, the above quantity is less than $C_2\int_{[-\infty,-\delta]\cup[\delta,\infty]}\psi_\delta(x_1)\dd x_1\leq  C_1C_2\delta$.
	
	Therefore, we have
	\begin{align*}
	\Res(| f_\delta- f_o|)(t)=O(\delta|K|+O_\delta/t)\| f\|_\lf,
	\end{align*}
	which completes the proof of Proposition \ref{prop:rescut}.
\end{proof}
\subsection{Proof of the Renewal theorem for stopping time}

\label{sec:proofrenewal}

Let us finally complete the proofs of Proposition \ref{prop:renewal} and Proposition \ref{prop:renewalquantitative} using Proposition \ref{prop:rescut}. Here we add the assumption that $\lambda$ is finitely supported.
\begin{proof}[Proof of Proposition \ref{prop:renewal}]
	Let $\rho$ be a smooth cutoff such that $\rho_{[0,|\supp\lambda|]}=1$ and becomes $0$ outside of $[-1,|\supp\lambda|+1]$.
	Take $f(y,x)=g(y+x)\rho(y)\rho(x+y)$. Then $f(y,x)=g(y+x)$ when $y$ and $x+y$ are in the interval $[0,|\supp\lambda|]$. By definition and $\supp\lambda\subset \R^+$, we have
	\[\E(g(S_{n_t}-t))=E_C f(t). \]
	This function $f$ satisfies the conditions in Proposition \ref{prop:rescut}, and $\|f\|_{\lf}\leq 8\|g\|_{C^1}$. The proof is complete by using Proposition \ref{prop:rescut}.
\end{proof}
\begin{proof}[Proof of Proposition \ref{prop:renewalquantitative}]
	We need to use the weakly diophantine condition to give an estimate of the error term in Proposition \ref{prop:rescut}. For the supremum of the absolute value of $u(i\xi)=\frac{1}{1-\cL\lambda(i\xi)}-\frac{1}{\sigma}\frac{1}{i\xi}$ and its derivative
	$$\partial_\xi\Big(\frac{1}{1-\cL\lambda(i\xi)}-\frac{1}{\sigma}\frac{1}{i\xi}\Big)=\frac{-\partial_\xi\cL\lambda(i\xi)}{(1-\cL\lambda(i\xi))^2}+\frac{1}{\sigma}\frac{1}{i\xi^2},$$
	on the interval $[-\delta^{-2},\delta^{-2}]$,
	by the definition $l$-weakly diophantine, we obtain that it is less than $C\delta^{-4l}$. Then by Proposition \ref{prop:rescut}
	\[O_\delta\leq C\delta^{-4l}. \]
	Then take $\delta=t^{-1/(4l+1)}$. The proof is complete.
\end{proof}

\section*{Acknowledgements} 
The first author would like to thank Jean-Fran\c{c}ois Quint for inspiring discussions on the regularity of self-similar measures. 
The second author thanks Pablo Shmerkin and Boris Solomyak for useful discussions. We also thank the anonymous referees for many useful comments and suggestions that greatly improved the presentation of this article. This work was completed while the second author was visiting Institut de Math\'ematiques de Bordeaux, and the authors would like to thank the hospitality of the institution.

\bibliographystyle{plain}

\begin{thebibliography}{10}

\bibitem{Baker} A. Baker: \textit{Transcendental Number Theory}, Cambridge University Press, 1st ed., 1975.

\bibitem{BHR} B. B\'ar\'any, M. Hochman, A. Rapaport: Hausdorff dimension of planar self-affine sets and measures. \textit{Invent. Math.} 216 (2019), 601--659

\bibitem{BG} J. Blanchet and P. Glynn, Uniform renewal theory with applications to expansions of random
geometric sums, \textit{Adv. in Appl. Probab.} 39 (2007), no. 4, 1070–1097. MR 2381589 (2009e:60191)

\bibitem{Bourgain2010}
J. Bourgain, The discretized sum-product and projection theorems, \textit{J. Anal. Math.} 112(2010), 193–236.

\bibitem{BourgainDyatlov}
J. Bourgain and S. Dyatlov.
\newblock Fourier dimension and spectral gaps for hyperbolic surfaces.
\newblock \textit{Geom. Funct. Anal.} 27 (2017), no. 4, 744-771. 

\bibitem{Boyer} J.-P. Boyer: The speed of convergence in the renewal theorem, Preprint, 2015, arXiv:1506.07625

\bibitem{Bre19} J. Br\'emont. Self-similar measures and the Rajchman property. Preprint, 2019, arXiv:1910.03463

\bibitem{BV} E. Breuillard, P. Varj\'u: On the dimension of Bernoulli convolutions. \textit{Ann. Prob.}, Volume 47, Number 4 (2019), 2582-2617.

\bibitem{Bufetov} A. Bufetov and B. Solomyak. On the modulus of continuity
for spectral measures in substitution dynamics. \textit{Adv. Math.} 260 (2014), 84–129

\bibitem{Cantor} G. Cantor. \textit{Gesammelte Abhandlungen mathematischen und philosophischen Inhalts}. Springer, 1932.

\bibitem{Carlsson} H. Carlsson. Remainder Term Estimates of the Renewal Function. \textit{Ann. Probab.} Volume 11, Number 1 (1983), 143-157.

\bibitem{Dai1} X-R. Dai: When does a {B}ernoulli convolution admit a spectrum? \textit{Adv. Math.} 231(3-
4):1681–1693, 2012.

\bibitem{Dai2} X-R.  Dai, D-J. Feng, Y. Wang: Refinable functions with non-integer dilations. \textit{J. Funct. Anal.}, 250(1):1–20, 2007.

\bibitem{DEL}
H. Davenport, P. Erd\"{o}s, and W. LeVeque.
\newblock On Weyl’s criterion for uniform distribution.
\newblock Michigan Math. J., 10:311-314, 1963.

\bibitem{Falconer1}
{K. Falconer: {Fractal geometry.}
	Mathematical foundations and applications. Second edition. John Wiley \& Sons, Inc., Hoboken, NJ, 2003. 
}


\bibitem{FengLau}D.-J. Feng, K.-S. Lau. Multifractal formalism for self-similar measures with
weak separation condition \textit{J. Math. Pures Appl.}, 92 (2009) 407–428.

\bibitem{Feller} William Feller, \textit{An introduction to probability theory and its applications. Vol. II.}, Second
edition, John Wiley \& Sons, Inc., New York-London-Sydney, 1971. MR 0270403 (42 \#5292)

\bibitem{GMSZ} X. Gao, J. Ma, K. Song, Y. Zhang. Fourier decay rate of coin-tossing type measures. \textit{J. Math. Anal. Appl.} 484(1) (2020)

\bibitem{FOS} J. Fraser, T. Orponen, T. Sahlsten: On the Fourier analytic properties of graphs. \textit{Int. Math. Res. Not.} 2014:10 (2014), 2730–2745

\bibitem{FS} J. Fraser, T. Sahlsten: On the Fourier analytic structure of the Brownian graph. \textit{Analysis \& PDE}, Volume 11, Number 1 (2018), 115-132.

\bibitem{Erdos} P. Erd\"os. On the smoothness properties of a family of Bernoulli convolutions. \textit{Amer. J. Math.}, 62:180–186, 1940.

\bibitem{Hochman} M. Hochman: On self-similar sets with overlaps and inverse theorems for entropy. \textit{Ann. of Math.} 180 (2014), no. 2, 773-822

\bibitem{HochmanShmerkinEquidistribution}
M. Hochman, P. Shmerkin. Equidistribution from fractal measures. \textit{Invent Math.} Volume 202, Issue 1, pp 427–479, 2015.

\bibitem{JordanSahlsten}
T Jordan, T. Sahlsten.
\newblock Fourier transforms of Gibbs measures for the Gauss map.
\newblock {\em Math. Ann. (2016) Vol 364 (3)}. 983-1023, 2015.

\bibitem{Kahane} J.-P. Kahane: Sur la distribution de certaines series aleatoires. In Colloque de Theorie des Nom-
bres (Univ. Bordeaux, Bordeaux, 1969), pages 119–122. Bull. Soc. Math. France, Mem. No. 25,
Soc. Math. France Paris, 1971.

\bibitem{KahaneImage}
J.-P. Kahane:
\newblock  Some Random series of functions (2nd ed.).
\newblock {\em  Cambridge University Press}. 1985.

\bibitem{KahaneLevel}
J.-P. Kahane:
\newblock  Ensembles alatoires et dimensions.
\newblock {\em  In Recent Progress in Fourier Analysis (proceedings of a seminar held in El Escorial). North-Holland, Amsterdam}. 65-121, 1983.

\bibitem{Kaufman1}
R. Kaufman:
\newblock  Continued fractions and Fourier transforms.
\newblock {\em  Mathematika, 27(2)}. 262-267, 1980.

\bibitem{Kaufman2}
R. Kaufman:
\newblock On the theorem of Jarn\'{i}k and Besicovitch.
\newblock {\em  Acta Arith.}. 39(3):265-267, 1981.

\bibitem{Kesten}
H. Kesten. Renewal Theory for Functionals of a Markov Chain with General State Space \textit{Ann. Probab.}
	Volume 2, Number 3 (1974), 355-386.
	
\bibitem{kershner}  R. Kershner: On singular Fourier-Stieltjes transforms. \textit{Amer. J. Math.} 58(2), 450--452 (1936)

\bibitem{KL92} A. S. Kechris and  A. Louveau: Descriptive set theory and harmonic analysis. \textit{The Journal of Symbolic Logic}. Volume 57, Issue 2 June 1992 , pp. 413-441

\bibitem{LP09}
I. Laba and M. Pramanik. 
Arithmetic progressions in sets of fractional dimension. \textit{Geom. Funct. Anal.}, 19(2):429–456, 2009

\bibitem{Lalley}
S. Lalley. Renewal theorems in symbolic dynamics, with applications to geodesic flows, noneuclidean tessellations and their fractal limits. \textit{Acta Math.}, Volume 163 (1989), 1-55.

\bibitem{Li1}
{J. Li:
	{Decrease of Fourier Coefficients of Stationary Measures.}
	\textit{Math. Ann.}, 372(3-4):1189--1238, 2018.
}

\bibitem{Li2}
J. Li: Fourier decay, Renewal theorem and Spectral gaps for random walks on split semisimple Lie groups. \textit{Preprint} (2018), arXiv:1811.06484

\bibitem{LSaffine} J. Li, T. Sahlsten: Fourier transform of self-affine measures. \textit{Adv. Math.}, 2020, to appear.

\bibitem{Mattila}
P. Mattila: \textit{Fourier Analysis and Hausdorff Dimension}, Cambridge University Press, 2015.

\bibitem{Menshov} D. Mensov. Sur l'unicit\'e du d\'evelloppement trigonom\'etrique. \textit{CRASP}, 163:433-436, 1916.

\bibitem{PS} Piatetski-Shapiro. \textit{Moscov. Gos. Univ. Uc. Zap.}, 165, Mat 7:79-97, 1954.

\bibitem{QetR}
M. Queff\'{e}lec and O. Ramar\'{e}.
\newblock Analyse de Fourier des fractions continues \'{a} quotients restreints.
\newblock Enseign. Math.(2), 49(3-4):335-356, 2003.

\bibitem{Riemann} B. Riemann: \textit{Habilitatsionschrift}. Abh. der Ges. der Wiss. zu Gott., 13:87- 132, 1868.

\bibitem{SahlstenStevens}
T. Sahlsten, C. Stevens:
\newblock{Fourier decay in nonlinear dynamics.}
\newblock {\em Preprint (2018)}, arXiv:1810.01378

\bibitem{Salem} R. Salem: Sets of uniqueness and sets of multiplicity. \textit{Trans AMS}, 54:218-228, 1943. Corrected on pp. 595-598, Trans AMS, vol 63, 1948.

\bibitem{SZ} R. Salem, A. Zygmund. Sur un theoreme de Piatetski-Shapiro. \textit{CRASP}, 240:2040-2042, 1954.

\bibitem{Sarnak} P. Sarnak: Spectra of singular measures as multipliers on $L^p$, \textit{J. Funct. Anal.} 37 (1980), 302–317.

\bibitem{ShmerkinGAFA} P. Shmerkin: On the exceptional set for absolute continuity of Bernoulli convolutions. \textit{Geom. Funct. Anal.} 24 (2014), no. 3, 946--958

\bibitem{ShmerkinSuomala}
P. Shmerkin and V. Suomala: Spatially independent martingales, intersections, and applications. Mem. Amer. Math. Soc. 251 (2018), no. 1195

\bibitem{Sol19} B. Solomyak. Fourier decay for self-similar measures. Preprint, 2019. arXiv:1906.12164

\bibitem{SidorovSolomyak} N. Sidorov, B. Solomyak: Spectra of Bernoulli convolutions as multipliers in $L^p$ on the circle. \textit{Duke Math. J.} Volume 120, Number 2 (2003), 353-370.

\bibitem{Strichartz1} R. Strichartz: Self-Similar Measures and Their Fourier Transforms I. \textit{Indiana University Mathematics Journal}
Vol. 39, No. 3 (Fall, 1990), pp. 797-817 (21 pages)

\bibitem{Strichartz2} R. Strichartz: Self-Similar Measures and Their Fourier Transforms II. \textit{Transactions of the
American Mathematical Society}. Volume 336, Number 1, 1993.

\bibitem{Tsujii} M. Tsujii: On the Fourier transforms of self-similar measures. \textit{Dynamical Systems: An International Journal},
Volume 30, Issue 4: Pages 468-484, 2015.

\bibitem{V} P. Varj\'u: On the dimension of Bernoulli convolutions for all transcendental parameters. \textit{Ann. of Math.} (2020), to appear

\bibitem{VY20} P. Varj\'u, H. Yu: Fourier decay of self-similar measures and self-similar sets of uniqueness. \textit{Analysis \& PDE} (2020), to appear, arXiv:2004.09358

\bibitem{Young} W. Young. A note on trigonometrical series. \textit{The Messenger of Mathematics}, 38:44-48, 1909.


\end{thebibliography}

\end{document}